\documentclass{amsart}
\usepackage{amsthm,amssymb,mathtools}
\usepackage{stmaryrd,tikz,pgfplots}
\usepackage{marginnote}

\usepackage[textsize=tiny]{todonotes}

\newcommand{\Marginpar}[1]{\marginpar{\tiny{#1}}}
\newcommand{\Note}[1]{{\par\noindent\hrulefill\par\tiny{#1}\par\noindent\hrulefill\par}}
\newcommand{\Detail}[1]{{#1}}
\renewcommand{\Marginpar}[1]{}
\renewcommand{\Note}[1]{}
\renewcommand{\Detail}[1]{}

\usepackage{hyperref}
\usepackage[all]{xy}
\usepackage{amsmath,amsfonts}	
\usepackage{amsthm,latexsym,amssymb}
\usepackage{tensor}

\newtheorem{thm}{Theorem}[section]
\newtheorem{prop}[thm]{Proposition}

\newtheorem{cor}[thm]{Corollary}

\theoremstyle{definition}

\renewcommand{\[}{\begin{equation*}}
\renewcommand{\]}{\end{equation*}}

\begin{document}
\parskip1mm

\title[]{$4$-dimensional almost-K\"ahler Lie algebras of constant Hermitian holomorphic sectional curvature are K\"ahler}

\author{Mehdi Lejmi}
\address{Department of Mathematics, Bronx Community College of CUNY, Bronx, NY 10453, USA.}
\email{mehdi.lejmi@bcc.cuny.edu}
\author{Luigi Vezzoni}
\address{Dipartimento di Matematica G. Peano \\ Universit\`a di Torino\\
Via Carlo Alberto 10\\
10123 Torino\\ Italy}
\email{luigi.vezzoni@unito.it}

\begin{abstract}
We prove that any $4$-dimensional almost-K\"ahler Lie algebra of constant Hermitian holomorphic sectional curvature
with respect to the canonical Hermitian connection is K\"ahler.
\end{abstract}
\maketitle


\section{Introduction}
An almost-Hermitian manifold is an even dimensional smooth manifold $M$ equipped with an almost complex structure $J$ and a compatible Riemannian metric $g$.
When the induced $2$-form $\omega(\cdot,\cdot)=g(J\cdot,\cdot)$ is closed the quadruple $(M,\omega,J,g)$ is called {\em almost-K\"ahler} and when $J$  is further integrable
it is called {\em K\"ahler}. 

A well-known conjecture of Goldberg~\cite{MR0238238} states that compact almost-K\"ahler Einstein manifolds are neceserly  K\"ahler. 
The conjecture is still open, but some partial results are known. 
Sekigawa~\cite{MR905633} proved that the conjecture is true
if the scalar curvature is non-negative and other results have been proved under various conditions specially in dimension $4$ (see for instance~\cite{MR1604803,MR1880825,MR1607545}).
Moreover, the conjecture is false if the compact is dropped since Nurowski and Przanowski constructed in ~\cite{MR1682582}  a non-compact example of almost-K\"ahler non-K\"ahler Ricci flat manifold which turns out to be of pointwise positive constant holomorphic sectional curvature. Here we recall that an almost-K\"ahler manifold is {\em of pointwise constant holomorphic sectional curvature} 
if, at any point of $M$, $g(R^g_{X,JX}X,JX)$ is constant for any vector field $X$ of unit length with respect to $g$, where $R^g$ denotes the Riemannian curvature of the Levi-Civita connection of $g$. Also, Sato~\cite{MR1422638,MR1700598,MR1806472} studied $4$-dimensional almost-K\"ahler manifolds of pointwise constant holomorphic sectional curvature. In particular, he 
constructed a non-compact example of pointwise negative constant holomorphic sectional curvature~\cite{MR1806472}.
Furthermore, almost-K\"ahler manifolds of pointwise constant  totally-real sectional curvature were examined~\cite{MR1632688,MR1638199,MR1782093}.
In dimension $4$, it turns out that such manifolds are self-dual and are conjectured to be K\"ahler in the compact case~\cite{MR1782093}.

On an almost-K\"ahler non-K\"ahler manifold $(M,\omega,J,g)$ the Levi-Civita connection does not preserve the almost-complex structure and its role is usually replaced 
by the canonical Hermitian connection~\cite{MR0066733,MR1456265} which is the defined as the unique connection $\nabla$ preserving $J$ and $g$ and having $J$-anti-invariant torsion. Hence it is natural to study the condition on a $4$-dimensional almost-K\"ahler manifold of having  pointwise constant holomorphic sectional  with respect to the canonical Hermitian connection instead of the Levi-Civita connection.  We refer to these manifolds as {\em of pointwise constant Hermitian holomorphic sectional curvature}.

The study of such manifolds was initiated
in~\cite{Lejmi-Markus} where it is proved that $4$-dimensional almost-K\"ahler manifolds of pointwise constant Hermitian holomorphic sectional curvature 
are self-dual and that in the compact case compact when the sectional curvature is globally non-negative constant the manifold is K\"ahler. 

In analogy to the Riemannian case, it is rather natural to pursuit the study of these manifolds in the non-compact case. In this note we study left-invariant almost-K\"ahler structures of having pointwise constant Hermitian holomorphic sectional curvature on Lie groups.  

\begin{thm}\label{main}
Any left-invariant almost-K\"ahler structure of pointwise constant
Hermitian holomorphic sectional curvature on a $4$-dimensional Lie group is K\"ahler.
\end{thm}

Since we assume the  almost-K\"ahler structure {\em invariant}, we can work on Lie algebras in an algebraic fashion.  In view of \cite{Lejmi-Markus} we can restrict our attention to $4$-dimensional Lie algebras with a self-dual metric. $4$-dimensional Lie algebras with a non-conformally flat self-dual metric are completely described in \cite{MR1955629} and from the classification it turns out that all the compatible almost-K\"ahler structures are K\"ahler. For the conformally flat case we use a classification in \cite{MR3442771} and we give a description of all the conformally flat metrics in terms of the existence of a special coframe. From our description  and the classification of symplectic Lie algebras \cite{MR2307912} it follows every conformally flat almost-K\"ahler Lie algebra of pointwise constant Hermitian holomorphic sectional curvature is K\"ahler. 

\subsection*{Acknowledgements}
Both authors are grateful to Anna Fino, Gueo Grantcharov, Simon Salamon for useful discussions.
The first named author was supported in part by a PSC-CUNY award \#60053-00 48, jointly funded by The Professional Staff Congress and The City University of New York.
The first named author thanks all the members of the Math Department of University of Turin for their warm hospitality during his stay in June 2017.
\section{Preliminaries}
%
On an almost-K\"ahler manifold $(M,\omega,J,g)$ the {\em canonical Hermitian connection} $\nabla$ is defined by the following formula ~\cite{MR0066733,MR1456265} 
\[
\nabla_XY=D^g_XY-\frac{1}{2}J\left( D^g_XJ\right)Y,
\]
where $D^g$ is the Levi-Civita connection of the metric $g$. We denote by $R^\nabla$ and $R^g$ the curvatures of $\nabla$ and $D^g$, respectively. We adopt the convention
 $R^\nabla_{X,Y}=-[\nabla_X,\nabla_Y]+\nabla_{[X,Y]}$, where $X,Y$ are vector fields on $M$ and $[\cdot,\cdot]$ is the commutator.
The {\em Hermitian homolomorphic sectional curvature}  $H$ is defined by 
\[
H(X):=\frac{g(R^\nabla_{X,JX}X,JX)}{g(X,X)g(X,X)}.
\]
We say that the Hermitian holomorphic sectional curvature is pointwise constant $k$ if 
\[
g(R^\nabla_{X,JX}X,JX)=k,
\]
at any point $p\in M$ and for any vector field $X$ of unit length with respect to $g.$

Now we focus on dimension $4$. In this case the bundle of $2$-forms $\Lambda^2M$
admits, under the action of the (Riemannian) Hodge operator, the $g$-orthogonal splitting
\[
\Lambda^2M=\Lambda^+M\oplus \Lambda^-M,
\]
where $\Lambda^+ M$ and $\Lambda^-M$ correspond to the bundles of self-dual and anti-self-dual $2$-forms, respectively.
The Riemannian curvature $R^g,$ viewed as a symmetric linear operator of $2$-forms $\Omega^2(M)=\Omega^+(M)\oplus\Omega^-(M)$, decomposes as 
\begin{equation}\label{curvature_operator}
R^g=\left(\begin{array}{cc}W^++\frac{s^g}{12}\mathrm{Id}|_{\Omega^+(M)} & \tilde{r}_0|_{{\Omega^-(M)}} \\ \tilde{r}_0|_{{\Omega^+(M)}} & W^-+\frac{s^g}{12}\mathrm{Id}|_{\Omega^-(M)}\end{array}\right),
\end{equation}
where $W^+$ and $W^-$ are called the self-dual Weyl tensor and the anti-self-dual Weyl tensor, respectively. Both $W^+$ and $W^-$ are trace-free. Here $s^g$ denotes the Riemannian scalar curvature
and $\tilde{r}_0$ is the Kulkarni--Nomizu extension of the trace-free part $r_0=r^g-\frac{s^g}{4}g$ of the Riemannian Ricci tensor $r^g.$
The manifold $(M,g)$ is said to be self-dual (anti-self-dual) if $W^-=0$ ($W^+=0$). $(M,g)$ is conformally flat if $W^+=W^-=0.$
The Riemannian curvature $R^g$ can also be viewed as a $(4,0)$-tensor via the metric, as well as the Weyl tensor $W=W^++W^-.$
Note also the decomposition of $W^+$ in the almost-K\"ahler case (see~\cite{MR1604803} or \cite{MR1921552} for more details)
\begin{equation}\label{weyl}
W^+=\left(\begin{array}{cc}\|N_J\|^2+\frac{s^g}{6} & W_\omega^+ \\ W_\omega^+ & W^+_{00}-\frac{1}{2}\left( \|N_J\|^2+\frac{s^g}{6}  \right)\mathrm{Id}|_{\Omega^{J,-}(M)}\end{array}\right),
\end{equation}
where $\Omega^{J,-}(M)$ are the $J$-anti-invariant $2$-forms on $M$ and $N_J$ is the Nijenhuis tensor of $J.$
It turns out that if a $4$-dimensional almost-K\"ahler manifold $(M,\omega,J,g)$ is of pointwise constant Hermitian holomorphic sectional curvature, then $W^-=0$~\cite{Lejmi-Markus}.
\begin{thm}\cite{Lejmi-Markus}\label{self_dual}
Let $(M,\omega,J,g)$ be an almost-K\"ahler manifold of dimension $4$. If $M$ is of pointwise constant Hermitian holomoprhic sectional curvature then $(M,g)$ is self-dual.
\end{thm}
Remark that the compactness assumption is not needed in the above Theorem. Furthemore, M. Upmeier and the first named author~\cite{Lejmi-Markus}
proved in the compact case that if the sectional curvature is globally constant and if $k$ is non-negative then $(M,\omega,J,g)$ is K\"ahler.

\section{Lie algebras admitting non-conformally flat self-dual metrics}\label{self-dual}
In this section we study $4$-dimensional almost-K\"ahler self-dual Lie algebras.

In~\cite{MR1955629}, De Smedt and Salamon proved the a Lie algebra $\mathfrak{g}$ of a $4$-dimensional Lie group admitting a non-conformally flat left-invariant metric $g$
such that $W^+=0$ has a coframe $\{e^1,e^2,e^3,e^4\}$ satisfying the following structure equations  
\begin{eqnarray*}
de^1&=&0,\\
de^2&=&-e^{12}-\lambda e^{13},\\
de^3&=&\lambda e^{12}-e^{13},\\
de^4&=&-2e^{14}+e^{23},
\end{eqnarray*}
for some $\lambda\geq 0$, where $e^{12}=e^1\wedge e^2,$ etc. Furthermore,
the metric $g$ has the (oriented) orthonormal basis $\{\frac{1}{k}e_1,e_2,e_3,e_4\}$ for $k=1$ or $k=2,$ where $\{e_i\}$ is the dual basis of $\{e^i\}$.
We remark that all the Lie algebras of this kind are not unimodular and, consequently, their Lie groups do not admit compact quotients.
To get non-conformally
left-invariant metrics such that $W^-=0$, we only have to change the orientation and consider the opposite orientation given by the volume form 
\begin{equation}\label{orientation}
-|k|e^1\wedge e^2\wedge e^3\wedge e^4.
\end{equation}

Now, we will find the almost-K\"ahler structures on $\mathfrak{g}$ with $g$ as a Riemannian metric and compatible with the opposite orientation~(\ref{orientation}). A 2-form $\omega=\displaystyle\sum_{1\leq i<j\leq 4}\alpha_{ij}e^{ij}$ on $\mathfrak{g}$ is closed if 
\begin{eqnarray*}
d\omega&=&\alpha_{14}\left(-e^1\wedge e^{23}\right)+\alpha_{23}\left(-e^{12}\wedge e^{3}\right)+\alpha_{23}\left(-e^{2}\wedge -e^{13}\right)\\
&+&\alpha_{24}\left(\left(-e^{12}-\lambda e^{13}\right)\wedge e^{4}\right)+\alpha_{24}\left(-e^{2}\wedge -2e^{14}\right)\\
&+&\alpha_{34}\left(\left(\lambda e^{12}-e^{13}\right)\wedge e^{4}\right)+\alpha_{34}\left(-e^{3}\wedge -2e^{14}\right)\\
&=&e^{123}(-\alpha_{14}-2\alpha_{23})+e^{124}(-3\alpha_{24}+\lambda\alpha_{34})+e^{134}(-\lambda\alpha_{24}-3\alpha_{34}),\\
&=&0.
\end{eqnarray*}
We deduce that 
$$
-2\alpha_{23}=\alpha_{14},\quad 
\lambda\alpha_{34} =3\alpha_{24},\quad 
\lambda\alpha_{24}=-3\alpha_{34}.
$$
We conclude that $\omega$ is closed if and only $\alpha_{34}=\alpha_{24}=0$ and $\alpha_{14}=-2\alpha_{23}.$
Hence, $\omega$ is closed and non-degenerate if it is of the form $$\omega=ae^{12}+be^{13}-2ce^{14}+ce^{23}$$
for some $a,b,c\in\mathbb{R}$ with $c\neq 0.$
Now, one can easily check that $\omega$ is compatible with the metric $g$ if and only if $a=b=0, c^2=1$ and $k^2=4.$ As a result, we get the following
\begin{cor}\label{non_confor_flat}
Any $4$-dimensional almost-K\"ahler Lie algebra with a self-dual non-conformally flat metric is K\"ahler. 
\end{cor}
\begin{proof}
From the above discussion, we obtain that, for such almost-K\"ahler structures $(\omega,J,g)$, the $2$-form $\omega$ takes the following expression
\[
\omega=\mp 2\,e^{14}\pm \,e^{23}\,.
\]
It is also compatible with the opposite orientation~(\ref{orientation}). Moreover, since $g$ is the standard metric with respect to basis $\{e_i\}$, the almost-complex structure $J$ satisfies 
\[
Je_1=- 2e_4,\quad Je_2= e_3,
\]
or
\[
Je_1=2e_4,\quad Je_2=-e_3.
\]
Then a direct computation of the Nijenhuis tensor yields that $J$ is integrable and that $(\omega,J,g)$ is a K\"ahler structure, as required. 
\end{proof}
\section{Lie algebras with conformally flat left-invariant metrics}
In this section we study conformally flat almost-K\"aher structures on $4$-dimensional Lie algebras.

The spectrum of the Riemannian curvature operators on $4$-dimensional Lie algebras with a conformally flat metrics was studied in~\cite[Theorem 2]{MR3442771}.
It turns out that there are only eight models of $4$-dimensional Lie algebras admitting a conformally flat metric. Moreover, by using the classification of $4$-dimensional symplectic Lie algebras in~\cite{MR2307912}, we get that 
the only three conformally flat Lie algebras admitting symplectic  structures are the following 
\begin{itemize}
\item the abelian Lie algebra;
\item $\mathfrak{r}\mathfrak{r}_{3,0}:\quad$ $[e_1,e_3]=-e_2,\quad [e_2,e_3]=e_1;$
\item $\mathfrak{r}_2^\prime:\quad$ $[e_1,e_3]=e_3,\quad [e_1,e_4]=e_4,\quad [e_2,e_3]=e_4,\quad[e_2,e_4]=-e_3.$
\end{itemize}

The abelian Lie algebra and $\mathfrak{r}\mathfrak{r}_{3,0}$ are unimodular and the spectrum of the Riemannian curvature operator~(\ref{curvature_operator}) of their conformally flat metrics
is $\{0,0,0,0,0,0\}$. In particular in these algebras any conformally flat metric has vanishing scalar curvature.
\begin{prop}\label{abelian_case}
Any conformally flat almost-K\"ahler structure on 
$\mathfrak{r}\mathfrak{r}_{3,0}$ is K\"ahler.
\end{prop}
\begin{proof}
Let $(\omega,J,g)$ be a confomally flat almost-K\"ahler structure on $\mathfrak{r}\mathfrak{r}_{3,0}$. From the above discussion the Riemannian scalar curvature $s^g$ is vanishing. 
Now, from the decomposition~(\ref{weyl}), we have 
\[
\|W^+\|^2=2\|W^+_\omega\|^2+\|W_{00}^+\|^2+\frac{3}{2}\left(\|N_J\|^2+\frac{s^g}{6}  \right)^2.
\]
Since $g$ is conformally flat, $\|N_J\|^2+\frac{s^g}{6}=0$. So we deduce that $s^g=\|N_J\|^2=0$
and that $J$ is integrable.
\end{proof}

Now, we study the remaining Lie algebra $\mathfrak{r}_2^\prime$. This algebra is not unimodular and from the spectrum of the Riemannian curvature operator
we can deduce that every conformally flat metric on $\mathfrak{r}_2^\prime$ has negative scalar curvature.
In terms of $1$-forms the structure equations of $\mathfrak{r}_2^\prime$ take the following expression  
\begin{align*}
de^1&=0,\\
de^2&=0,\\
de^3&=-e^{13}+e^{24},\\
de^4&=-e^{23}-e^{14},
\end{align*}
where $\{e^i\}$ is a basis of $\left(\mathfrak{r}_2^\prime\right)^*$. 

We give a description of all the possible conformally flat metrics on $\mathfrak{r}_2^\prime.$ 
Let $g$ be an arbitrary metric on $\mathfrak{r}_2^\prime$. 
By applying Gram--Schmidt algorithm we construct from the basis $\{e^i\}$  a $g$-orthonormal basis $\{f^i\}$
of $\left(\mathfrak{r}_2^\prime\right)^*$ such that 
\begin{align*}
f^1&=a_1e^1,\\
f^2&=a_2f^1+a_3e^2,\\
f^3&=a_4f^1+a_5f^2+a_6e^3,\\
f^4&=a_7f^1+a_8f^2+a_9f^3+a_{10}e^4,
\end{align*}
for some $a_i\in \mathbb{R}$ with  $a_1,a_3,a_6,a_{10}>0$. By differentiating the coframe $\{f^i\}$ we get 
\begin{eqnarray*}
df^1&=&0,\\
df^2&=&0,\\
df^3&=&-a_6e^{13}+a_6e^{24}\\
df^4&=&a_9\left(-a_6e^{13}+a_6e^{24}\right)-a_{10}e^{23}-a_{10}e^{14}.
\end{eqnarray*}
And so we deduce the following structure equations with respect to the coframe $\{f^i\}$:
\begin{align*}
df^1&=0,\\
df^2&=0,\\
df^3&=\frac{a_1a_2a_6a_8+a_1a_6a_7+a_{10}a_3a_5}{a_1a_3a_{10}}f^{12}+\frac{a_1a_2a_6a_9-a_{10}a_3}{a_1a_3a_{10}}f^{13}-\frac{a_6a_2}{a_3a_{10}}f^{14}-\frac{a_6a_9}{a_3a_{10}}f^{23}+\frac{a_6}{a_3a_{10}}f^{24}\\
df^4&=\frac{a_1a_2a_6^2a_8a_9-a_1a_{10}^2a_2a_5+a_1a_6^2a_7a_9+a_{10}a_3a_5a_6a_9-a_1a_{10}^2a_4+a_{10}a_3a_6a_8}{a_1a_3a_6a_{10}}f^{12}\\
&+\frac{a_2a_6^2a_9^2+a_{10}^2}{a_3a_{10}a_6}f^{13}-\frac{a_1a_2a_6a_9+a_{10}a_3}{a_1a_3a_{10}}f^{14}-\frac{a_6^2a_9^2+a_{10}^2}{a_3a_{10}a_6}f^{23}+\frac{a_9a_6}{a_3a_{10}}f^{24}.
\end{align*}
We need to find all possible solutions for the vanishing of the Weyl tensor $W=W^++W^-$.
A direct computation gives 
$$
W\left(f_1,f_3,f_2,f_3\right)=\frac{a_2\left((a_9^2-1)a_6^2+a_{10}^2\right)\left((a_9^2+1)a_6^2+a_{10}^2\right)a_1-2a_{10}a_3a_6^3a_9}{a_1a_3^2a_{10}^2a_6^2}
$$
and so 
$$
W=0\Longrightarrow a_2\left((a_9^2-1)a_6^2+a_{10}^2\right)\left((a_9^2+1)a_6^2+a_{10}^2\right)a_1-2a_{10}a_3a_6^3a_9=0\,. 
$$
Notice that since $a_1,a_3,a_6,a_{10}>0$, $(a_9^2+1)a_6^2+a_{10}^2\neq 0.$ 
Assume $g$ conformally flat (i.e $W=0$).  
If we assume $(a_9^2-1)a_6^2+a_{10}^2\neq 0$, then we get 
\begin{equation}\label{a2}
a_2=\frac{2a_{10}a_3a_6^3a_9}{\left((a_9^2-1)a_6^2+a_{10}^2\right)\left((a_9^2+1)a_6^2+a_{10}^2\right)a_1},
\end{equation}
and we get assuming~(\ref{a2})
$$W(f_1,f_3,f_2,f_4)=\frac{\left(a_6^2a_9^2+a_{10}^2-2a_{10}a_6+a_6^2\right)\left(a_6^2a_9^2+a_{10}^2+2a_{10}a_6+a_6^2\right)}{-2a_3a_{10}(a_6^2a_9^2+a_{10}^2-a_6^2)a_6a_1}=0$$
Since $a_6^2a_9^2+a_{10}^2+2a_{10}a_6+a_6^2\neq 0$ we deduce  
$$
a_6^2a_9^2+a_{10}^2-2a_{10}a_6+a_6^2=a_6^2a_9^2+(a_{10}-a_6)^2=0\,. 
$$ 
Hence $a_{10}=a_6$ and $a_9=0.$
This is in contradiction with our hypothesis $(a_9^2-1)a_6^2+a_{10}^2\neq 0$ and we conclude that $(a_9^2-1)a_6^2+a_{10}^2= 0$. Thus,
\begin{equation}\label{a9}
a_9=\pm\sqrt{1-\frac{a^2_{10}}{a_6^2}}.
\end{equation}
We get assuming~(\ref{a9})
$$W(f_1, f_3, f_2, f_3)=\pm\frac{a_6\sqrt{ 1-\frac{a^2_{10}}{a^2_6} }}{a_1a_3a_{10}}=0$$
We deduce that (recall that $a_6,a_{10}>0$)
\begin{equation}\label{a10}
a_{10}=a_6.
\end{equation}
We obtain assuming~(\ref{a9}) and~(\ref{a10})
$$W(f_1, f_2, f_2, f_3)=\frac{a_1a_2a_5+a_1a_4-a_3a_8}{-4a_3^2a_1}=0.$$
Hence, 
\begin{equation}\label{a8}
a_8=\frac{a_1a_2a_5+a_1a_4}{a_3}.
\end{equation}
Assuming~(\ref{a9}),~(\ref{a10}) and~(\ref{a8}), we get 
$$W(f_2, f_4, f_3, f_4)=\frac{a_1^2a_2^2a_5+a_1^2a_2a_4+a_1a_3a_7+a_3^2a_5}{-4a_3^2a^2_1}=0.$$
We deduce that
\begin{equation}\label{a7}
a_7=\frac{-a_1^2a_2^2a_5-a_1^2a_2a_4-a_3^2a_5}{a_1a_3}.
\end{equation}
Under the conditions ~(\ref{a9}),~(\ref{a10}),~(\ref{a8}) and~(\ref{a7}), one can verify that the Weyl tensor vanishes and so we get 
\begin{prop}\label{conf_flat}
Conformally flat metrics on $\mathfrak{r}_2^\prime$ are described by the orthonormal basis $\{f^i\}$ of $\left(\mathfrak{r}_2^\prime\right)^*$
\begin{eqnarray*}
f^1&=&a_1e^1,\\
f^2&=&a_2f^1+a_3e^2,\\
f^3&=&a_4f^1+a_5f^2+a_6e^3,\\
f^4&=&a_7f^1+a_8f^2+a_9f^3+a_{10}e^4.
\end{eqnarray*}
($a_1,a_3,a_6,a_{10}>0$) subject to the relations
\begin{enumerate}
\item $a_6=a_{10}$,\\
\item $a_9=0,$\\
\item $a_8=\frac{a_1a_2a_5+a_1a_4}{a_3},$\\
\item $a_7=\frac{-a_1^2a_2^2a_5-a_1^2a_2a_4-a_3^2a_5}{a_1a_3}.$
\end{enumerate}
\end{prop}

Now, we consider a conformally flat almost-K\"ahler structures $(\omega,J,g)$ on $\mathfrak{r}_2^\prime$. 
The symplectic form can be written as 
\[
\omega=b_1f^{12}+b_2f^{13}+b_3f^{14}+b_4f^{23}+b_5f^{24}+b_6f^{34},
\]
where $b_i\in\mathbb{R}$ and the basis $\{f^i\}$ is a $g$-orthonormal basis of $\left(\mathfrak{r}_2^\prime\right)^*$ as described in Proposition~\ref{conf_flat}.
In order for $\omega$ to be closed it is easy to check that (keeping in mind that $a_1,a_3,a_6>0$)
\begin{enumerate}
\item $b_6=0,$\\
\item $b_5:=\frac{-a_1a_2a_6b_4-a_1a_6b_2}{a_3a_6},$\\
\item $b_4 =\frac{-a_1^2a_2b_2+a_1a_3b_3}{a_1^2a_2^2+a_3^2}.$
\end{enumerate}
and the 2-form $\omega$ writes as 
$$\omega=b_1f^{12}+b_2f^{13}+b_3f^{14}-\frac{a_1(a_1a_2b_2-a_3b_3)}{a_1^2a_2^2+a_3^2}f^{23}-\frac{a_1(a_1a_2b_3+a_3b_2)}{a_1^2a_2^2+a_3^2}f^{24}.$$
The matrix 
\[
J=\left[\begin{array}{cccc}0 & b_1 & b_2 & b_3 \\-b_1 & 0 & -\frac{a_1(a_1a_2b_2-a_3b_3)}{a_1^2a_2^2+a_3^2} & -\frac{a_1(a_1a_2b_3+a_3b_2)}{a_1^2a_2^2+a_3^2} \\-b_2 & \frac{a_1(a_1a_2b_2-a_3b_3)}{a_1^2a_2^2+a_3^2} & 0 & 0 \\-b_3 & \frac{a_1(a_1a_2b_3+a_3b_2)}{a_1^2a_2^2+a_3^2}  & 0 & 0\end{array}\right]
\]
has to satisfy $J^2=-\mathrm{Id}$, in order for $(\omega,J,g)$ to be an almost-K\"ahler structure.
From $J^2=-\mathrm{Id}$ we see that $b_1b_2=b_1b_3=0.$ If we suppose that $b_1\neq 0$ then $b_2=b_3=0$ but then $\omega\wedge\omega=0.$
Hence, we deduce that 
\begin{equation}\label{b1}
b_1=0.
\end{equation}
Assuming~(\ref{b1}), we get, from $J^2=-\mathrm{Id}$, that $$b_2^2+b_3^2=1.$$
We also obtain $$\frac{b_2a_1(a_1a_2b_2-a_3b_3)}{a_1^2a_2^2+a_3^2}+\frac{b_3a_1(a_1a_2b_3+a_3b_2)}{a_1^2a_2^2+a_3^2}=0.$$ It implies that $a_2a_1^2(b_2^2+b_3^2)=a_2a_1^2=0.$
We deduce that $a_2=0$ and then we get $a_3=a_1.$
\begin{prop}\label{conf_ak}
Conformally flat almost-K\"ahler structures $(\omega,J,g)$ on $\mathfrak{r}_2^\prime$ are described by 
$$g=f^1\otimes f^1+f^2\otimes f^2+f^3\otimes f^3+f^4\otimes f^4,$$
$$\omega=b_2f^{13}+b_3f^{14}-\frac{a_1(a_1a_2b_2-a_3b_3)}{a_1^2a_2^2+a_3^2}f^{23}-\frac{a_1(a_1a_2b_3+a_3b_2)}{a_1^2a_2^2+a_3^2}f^{24}.$$
The $\{f^i\}$ are described in Proposition~\ref{conf_flat}. The constants $a_1,a_2,a_3,a_4,a_5,a_6,b_2,b_3$ are subject to the relations
\begin{enumerate}
\item $a_3=a_1$;\\
\item $b_2^2+b_3^2=1;$\\
\item $a_2=0.$\\
\end{enumerate}
\end{prop}
All the conformally flat almost-K\"ahler structures given in Proposition~\ref{conf_ak} are non-K\"ahler. Indeed if we compute the Nijenhuis tensor we get for instance
$$N(f_1,f_2)=\frac{b_2^2+b_3^2}{2a_1}f_2=\frac{1}{2a_1}f_2\neq 0.$$
Now, we compute the holomorphic Hermitian sectional curvature of the unit vectors $f_i$ of the conformally flat almost-K\"ahler structures described in Proposition~\ref{conf_ak}.
We get that 
\begin{enumerate}
\item $H(f_1)=\frac{-1}{a_1^2}$\\
\item $H(f_2)=\frac{-1}{2a_1^2}$\\
\item $H(f_3)=-\frac{1+b_2^2}{2a_1^2}$\\
\item $H(f_4)=-\frac{1+b_3^2}{2a_1^2}$
\end{enumerate}
We conclude that none of those almost-K\"ahler structures is of pointwise constant Hermitian holomorphic sectional curvature and we get the following
\begin{prop}\label{last_case}
The Lie algebra $\mathfrak{r}_2^\prime$ has no conformally flat almost-K\"ahler structures of pointwise constant Hermitian holomorphic sectional curvature.
\end{prop}
We can now prove our main result by combining all the previous results.

\begin{proof}[Proof of Theorem $\ref{main}$]
Suppose that $\mathfrak{g}$ is a $4$-dimensional Lie algebra with an almost-K\"ahler structure $(\omega,J,g)$
of pointwise constant Hermitian holomorphic sectional curvature. It follows from Theorem~\ref{self_dual}
that $W^-=0.$ If the metric is not conformally flat then it is given in~\cite{MR1955629} and it follows from Corollary~\ref{non_confor_flat} that it is K\"ahler. If it is conformally flat then it is described in~\cite[Theorem 2]{MR3442771} and it follows
 from Corollary~\ref{abelian_case} and Proposition~\ref{last_case} that it is K\"ahler.
\end{proof}

\bibliographystyle{abbrv}

\bibliography{biblio-sectional_curv}

\end{document}